\documentclass[12pt]{article}

\usepackage{mathtools,cancel}

\usepackage{amsmath, amssymb,amsthm,graphicx,float,ifthen,placeins}
\usepackage[all]{xypic}
\usepackage{etex}
\usepackage{authblk}
\theoremstyle{definition}

\theoremstyle{definition}
\newtheorem{prob}{Problem}
\newtheorem{theorem}{Theorem}
\newtheorem{definition}[theorem]{Definition}

\title{Motif Patterns and Coverings of Points with  Unit Disks, Part I}
\author{Jeremy F. Alm\footnote{Corresponding author}, Nicholas Hommowun, Elizabeth \\ Manary, and Aaron Schneider}
\affil{Department of Mathematics\\ Illinois College\\ 1101 W. College Ave.\\Jacksonville, IL 62650}
\begin{document}

\maketitle
\begin{abstract}
We consider a modification of Winkler's  ``dots and coins" problem, where we constrain the dots to lie on a square lattice in the plane.  We construct packings of ``coins" (closed unit disks) using motif patterns.
\end{abstract}
\section{Winkler's problem}
In his ``Puzzled" column \cite{Winkler10}, Peter Winkler discusses the following problem:

\begin{quote}
  { What is the largest integer $k$ such that any $k$ points in the plane, no matter how they are arranged, can always be covered with disks with pairwise-disjoint interior having radius 1?}
\end{quote}

Winkler states \cite{Winkler10b} that there exists a constructive proof  that gives a covering for any set of 12 points. In \cite{Kiyomi2012}, the authors give a configuration of 53 points on a triangular lattice that cannot be covered with unit disks.  Hence $12\leq k< 53$. This is a challenging problem, one that is likely to generate some interesting mathematics.

We might consider modifying the problem by constraining the locations of the points in some way; for instance, they could be restricted to lying on a square or triangular lattice.  In this essay, we take up the question of covering the points of the square lattice:

\begin{quote}
  { For which $d>0$ is it possible to cover all the points of the square lattice with inter-point distance $d$ (i.e., $(d\mathbb{Z}) \times(d\mathbb{Z})$) with disks with pairwise-disjoint interior having radius 1?}
\end{quote}

Call this lattice $L_d$.  The principal result of this essay is the following:

\begin{theorem}
For all $d\in \left[\frac{2}{\sqrt{13}},{\frac{1}{\sqrt{2}}}\right]\bigcup \left[\frac{4}{\sqrt{26}},\infty\right)$, $L_d$ can be covered with unit disks with pairwise-disjoint interior.
\end{theorem}

Note: unlike in \cite{Kiyomi2012}, we consider  \emph{closed} disks of unit radius.

\begin{figure}[htb!]
\centering
\includegraphics[width=200pt]{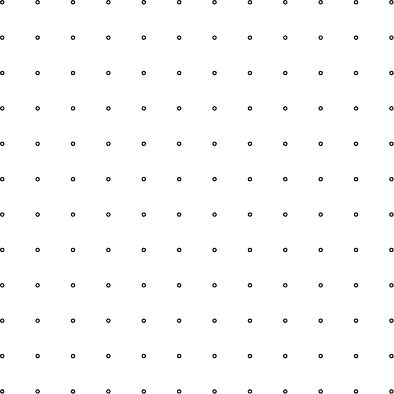}
\caption{$L_d$}
\label{fig:F1}
\end{figure}

%\begin{center}
%
%  \includegraphics*[width=150pt]{almgrid.pdf}
%
%
%
%
%\end{center}

%%%%%%%%%%%%%%%%%%%%%%%%%%%%%%%%%%%%

\section{Definitions, Theorems, and a Method}

It is  easy to see that if $d\geq 2$, we can cover all the points of $L_d$  by giving each point its own disk.  What about $d< 2$?

Consider  $d= \sqrt{2}$:  a circle with  unit radius circumscribes a square with  side length $\sqrt{2}$, so we can cover all the points, four at a time, with unit disks:

\begin{figure}[H]
\centering
\includegraphics[width=200pt]{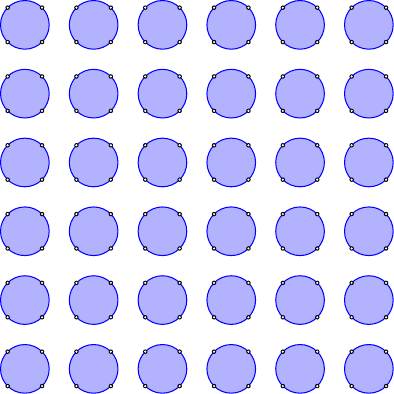}
\caption{A covering of $L_{\sqrt{2}}$}
\label{fig:F1}
\end{figure}

%\begin{center}
%
%  \includegraphics*[width=150pt]{alm9.pdf}
%
%
%
%
%\end{center}

It is not hard to see that this same strategy will work for smaller values of $d$; the only thing we  need to worry about is making sure the disks do not overlap. In fact, this configuration works for all $1\leq d\leq \sqrt{2}$. This line of reasoning will be made more rigorous in  Theorem \ref{th}.

Let's try to extract the essential features of the previous example:

\begin{itemize}
 \item    a finite number of points were selected;
 \item  all of those points were within one unit of the  ``center";
 \item   distinct ``centers"  were at least two units apart.

\end{itemize}

We   codify this method in the form of a theorem.  First we need a few definitions. The first is from \cite{Grunbaum}.

\begin{definition}
  A \textbf{motif} is a non-empty plane set. A \textbf{motif pattern} $\mathcal{P}$ \textbf{with  motif} $M$ is a non-empty family $\{M_i:i\in I\} $ such that \begin{enumerate}
\item  $\forall i , \ M_i\text{ is congruent to }M$;
\item $\forall i\neq j , \ M_i\cap M_j =\varnothing$;
\item $\forall i, j ,$ there exists an isometry of the plane mapping $\mathcal{P}$ onto itself and $M_i$ onto $M_j$.
\end{enumerate}

\end{definition}

\begin{figure}[htb!]
\centering
\includegraphics[width=200pt]{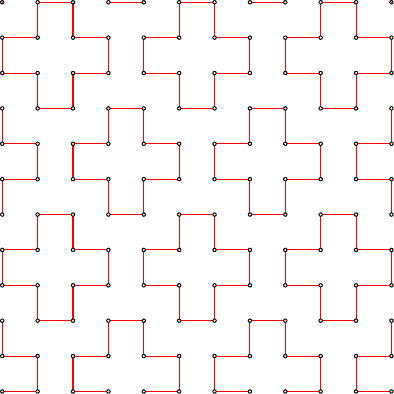}
\caption{A Motif Pattern}
\label{fig:F1}
\end{figure}

%\begin{center}
%
%  \includegraphics*[width=150pt]{alm7.pdf}
%
%
%
%
%\end{center}

\begin{definition}
 Let $\mathcal{P} $ be a motif pattern with motif $M$.  Let  $p$ be a point on $M$, called the \textbf{center}, and let $p_i$ be the corresponding (center) point on $M_i$.  Define $$\alpha=\min\{d (p_i,p_j ):i\neq j\} $$ and $$\beta=\max\{d (p,x ):x\in M\} .$$
Then the motif pattern $\mathcal{P}$ is  \textbf{admissible} if $2\beta\leq\alpha$.
\end{definition}

We consider only closed and bounded motif patterns.  See Figure \ref{fig:F4} for an illustration of an admissible motif pattern with $\beta$ indicated, as well as two candidates for $\alpha$. The following theorem is our ``workhorse":

\begin{figure}[htb!]
\centering
\includegraphics[width=200pt]{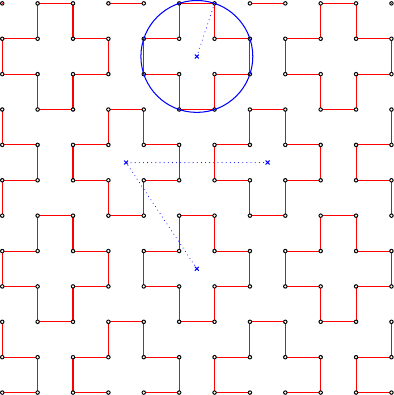}
\caption{An Admissible Motif Pattern with Centers Shown}
\label{fig:F4}
\end{figure}

%\begin{center}
%
%  \includegraphics*[width=210pt]{alm8.pdf}
%
%
%
%
%\end{center}

\begin{theorem}\label{th}
Let $\mathcal{P}$ be an admissible motif pattern.  Suppose that   $L_1\subseteq\mathcal{P}$, i.e. $\mathcal{P}$ covers all the points of $L_1$. Then $L_d$ can be covered by unit disks with pairwise-disjoint interior for all $d\in [2/\alpha, 1/\beta]$.
\end{theorem}

\begin{proof}
Let $\mathcal{P}$ be an admissible motif pattern with motif $M $ that covers every point on $L_1$.  Let $d\in [2/\alpha, 1/\beta]$.  Place a disk $D_i $ of radius $\frac{1}{d}$ at the center $p_i$ of each copy $M_i $ of $M $. We have, by assumption, that $$L_1\subset\bigcup_i M_i   .  $$ Since $\frac{1}{d}\geq\beta$, $$L_1\subset\bigcup_i M_i \subseteq \bigcup_i D_i .  $$  Since $\frac{2}{d}\geq\alpha$, distinct disks $D_i $ and $D_j $ have disjoint interiors.  Now dilate $L_1$ and each $D_i $ by a factor of $d $. Then we have a covering of $L_d$ by units disks, as desired.
\end{proof}

\section{Motif Patterns and  Coverings of $L_d$}

    In light of the previous theorem, what remains to be done is to find admissible motif patterns. In Figures 5-11 below, several motif patterns are given. The centers of the motifs are not indicated, since they are right where you think they should be: at the centers of mass of the motifs.
\begin{figure}[H]
\centering
\includegraphics[width=200pt]{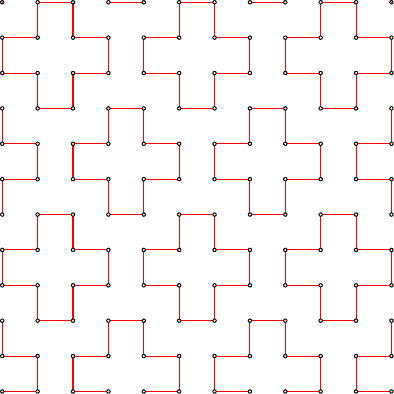}
\caption{A  Motif Pattern  for the Interval $\left[\frac{2}{\sqrt{13}},\sqrt{\frac{2}{5}}\right] $ }
\label{fig:F1}
\end{figure}
%\begin{center}
%
%  \includegraphics*[width=150pt]{alm7.pdf}
%
%
%
%
%\end{center}
%\begin{figure}[htb!]
%\centering
%\includegraphics[width=200pt]{alm.pdf}
%\caption{A  Motif Pattern  for the Interval }
%\label{fig:F1}
%\end{figure}

%\begin{center}
%
%  \includegraphics*[width=150pt]{alm6.pdf}
%
%
%
%
%\end{center}
\begin{figure}[H]
\centering
\includegraphics[width=200pt]{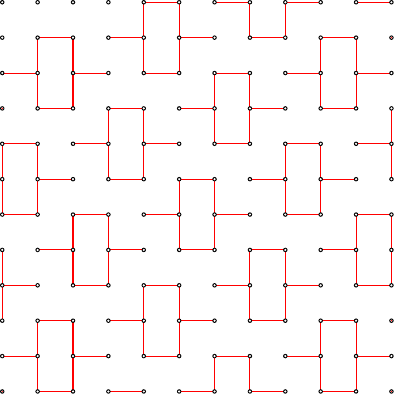}
\caption{A  Motif Pattern  for the Interval $\left[\sqrt{\frac{2}{5}},\frac{2}{3}\right] $}
\label{fig:F1}
\end{figure}

%\begin{center}
%
%  \includegraphics*[width=150pt]{alm5.pdf}
%
%
%
%
%\end{center}
\begin{figure}[H]
\centering
\includegraphics[width=200pt]{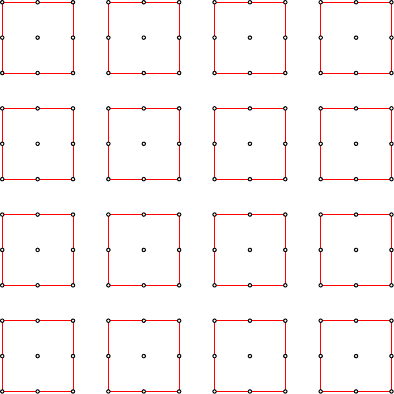}
\caption{A  Motif Pattern  for the Interval $\left[\frac{2}{3},\frac{1}{\sqrt{2}}\right] $}
\label{fig:F1}
\end{figure}

%\begin{center}
%
%  \includegraphics*[width=150pt]{alm4.pdf}
%
%
%
%
%\end{center}
\begin{figure}[H]
\centering
\includegraphics[width=200pt]{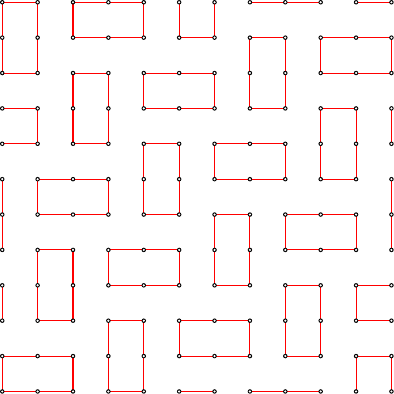}
\caption{A  Motif Pattern  for the Interval $\left[\frac{4}{\sqrt{26}},\frac{2}{\sqrt{5}}\right] $}
\label{fig:F1}
\end{figure}

%\begin{center}
%
%  \includegraphics*[width=150pt]{alm3.pdf}
%
%
%
%
%\end{center}
\begin{figure}[H]
\centering
\includegraphics[width=200pt]{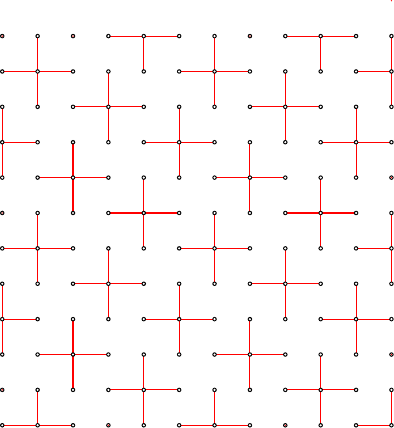}
\caption{A  Motif Pattern  for the Interval $\left[\frac{2}{\sqrt{5}},1\right] $}
\label{fig:F1}
\end{figure}

%\begin{center}
%
%  \includegraphics*[width=150pt]{alm2.pdf}
%
%
%
%
%\end{center}
\begin{figure}[H]
\centering
\includegraphics[width=200pt]{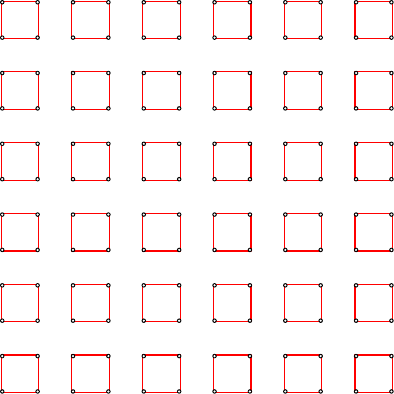}
\caption{A  Motif Pattern  for the Interval $[1,\sqrt{2}] $}
\label{fig:F1}
\end{figure}

%\begin{center}
%
%  \includegraphics*[width=150pt]{alm1.pdf}
%
%
%
%
%\end{center}
\begin{figure}[H]
\centering
\includegraphics[width=200pt]{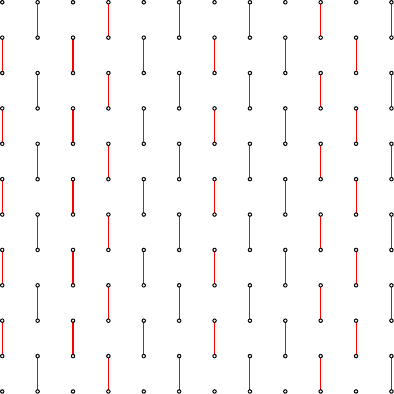}
\caption{A  Motif Pattern  for the Interval $[\sqrt{2},2] $}
\label{fig:F1}
\end{figure}

\section{Further Questions}

There is an annoying gap between $\frac{1}{\sqrt{2}}$ and $\frac{4}{\sqrt{26}}$, of width approximately 0.077, that contains 3/4. All of the authors believe that the set of all $d$ such that $L_d$ can be covered with unit disks is an interval, but we have no proof; nor do we have a way of closing this gap (yet).

\begin{prob}
  Is $L_\frac{3}{4}$ coverable?
\end{prob}

There is also the question of a lower bound.  In \cite{Kiyomi2012}, the authors argue that for $d<2(\frac{2\sqrt{3}}{3}-1)$, $L_d$ cannot be covered with disjoint unit disks.  The present authors believe the true lower bound to be closer to 1/2.

\begin{prob}
  Find a lower bound on $$\{d>0: L_d\text{ can be covered with unit disks with disjoint interior}\}.$$
\end{prob}

  Naturally, we could also consider the triangular lattice.

\begin{prob}
  Let $ T_d$ denote the triangular lattice with inter-point distance $d$. Determine for which $d>  0$  $ T_d$ can be covered with unit disks with disjoint interior.
\end{prob}

Now let us return to a question more in the spirit of the original problem.
\begin{prob}
  What is the largest integer $\ell$ so that for all $d>  0$, any set of  $\ell$ points on $L_d$ can be covered with unit disks with disjoint interior?
\end{prob}

By Winkler's argument, $\ell\geq 12$. In \cite{Kiyomi2012}, the authors note that their method, using the square lattice instead of the triangular lattice, gives a set of 102 points that cannot be covered.  So $12\leq\ell< 102$.

We will take up some of these questions in a subsequent essay.

%%%%%%%%%%%%%%%%%%%%%%%%%%%%%%%%%%%%
\section{Acknowledgments}

We gratefully acknowledge the support of the Illinois College Student-Faculty Research Fund.  We also thank Jacob Manske for rendering the figures for this paper.

%\bibliographystyle{plain}
%\bibliography{../masterrefs}

\end{document}